\documentclass[12pt,reqno]{amsart}
\newtheorem{lemma}{Lemma}%[section]

\newtheorem{theorem}[lemma]{Theorem}
\newtheorem{remark}{Remark}

\newtheorem{corollary}[lemma]{Corollary}
\newtheorem{example}{Example}
\usepackage{mathtools}

\newcommand{\llla}{\left\lbrace}
\newcommand{\rlla}{\right\rbrace}

\newcommand{\R}{\mathbb{R}}

\newcommand{\0}{\mathcal{O}}
\usepackage{amssymb,amsmath,amsthm}
\usepackage{epsfig,amssymb,amsmath,amsthm,graphicx,psfrag}
\usepackage[all]{xy}
\usepackage{enumerate}
\usepackage{color}
\usepackage{animate}
\usepackage{float}
\usepackage{subfig}
\begin{document}
\title{Transformation groups of certain flat affine manifolds}

% for the Lie group of affine transformations of the line

\author{Saldarriaga, O.*, and Fl\'orez, A.*}

\subjclass[2010]{Primary: 57S20, 54H15; Secondary: 53C07, 17D25  \\ Partially Supported by CODI, Universidad de Antioquia. Project Number 2015-7654.}
\date{\today}

\begin{abstract} %This paper deals with flat affine connections on paracompact real manifolds without boundary. 
In this paper we characterize the group of affine transformations of a flat affine simply connected manifold whose developing map is a diffeomorphism. This is proved by making use of some simple facts about homeomorphisms of $\mathbb{R}^n$ preserving open connected sets. We show some examples where the characterization is useful. 
\end{abstract}
\maketitle


\vskip5pt
\noindent
$*$ Instituto de Matem\'aticas, Universidad de Antioquia, Medell\'in-Colombia

e-mails: omar.saldarriaga@udea.edu.co, walexander.florez@udea.edu.co

\section{Introduction}
In what follows $M$ is a connected real n-dimensional manifold without boundary and  $\nabla$ a linear connection on $M$. The torsion  and curvature tensors of a connection $\nabla$ are defined as
\begin{align*} T_{\nabla}(X,Y)&=\nabla_XY-\nabla_YX-[X,Y] \\
K_{\nabla}(X,Y)&=[\nabla_X,\nabla_Y]  - \nabla_{[X,Y]}
\end{align*}
 for any $X,Y\in\mathfrak{X}(M)$, where $\mathfrak{X}(M)$ denotes the space of smooth vector fields  on $M$. If the curvature and torsion tensors of $\nabla$ are both null, the connection is called flat affine and the pair $(M,\nabla)$ is called a flat affine manifold. An affine transformation of $(M,\nabla)$ is a diffeomorphism $F$ of $M$  verifying the system of differential equations $F_*\nabla_XY=\nabla_{F_*X}F_*Y$, for all $X,Y\in \mathfrak{X}(M).$ The group of affine transformations $\it{Aff}(M,\nabla)$ is the set of diffeomorphisms  of $M$ preserving $\nabla$. This set endowed with the open-compact topology and composition is a Lie group (see \cite{KoNo} page 229). 
 
 The usual connection $\nabla^0$ in $\mathbb{R}^n$ is defined, in local coordinates, by $\nabla_X^0Y=\displaystyle{\sum_{i=1}^nX(f_i)\dfrac{\partial}{\partial x^i}}$ when $Y=\displaystyle{\sum_{i=1}^nf_i\dfrac{\partial}{\partial x^i}}.$ The group $\it{Aff}(\mathbb{R}^n,\nabla^0)$ agrees with the group of classical affine transformations $\it{Aff}(\mathbb{R}^n)$, that is, diffeomorphisms of $\mathbb{R}^n$  obtained as a composition of a linear transformatioans map defined by $w\mapsto T(w)+v$ (by linear transformation we mean a linear isomorphism).  We will also use the matrix notation $[\begin{matrix} A&v'\end{matrix}]$ with  $A$ the matrix of the linear transformation  and $v'$ the coordinate vector of the translation under some basis. 
  
An infinitesimal affine transformation of $(M,\nabla)$ is a smooth vector field $X$ on $M$ whose local 1-parameter groups $\phi_t$ are local affine transformations of $(M,\nabla)$.   We will denote by $\mathfrak{a}(M,\nabla)$ the real vector space of infinitesimal affine transformations of $(M,\nabla)$. An element $X$ of $\mathfrak{X}(M)$ belongs to $ \it{aff}(M,\nabla)$ if and only if it verifies
\[ \mathcal{L}_X\circ \nabla_Y -\nabla_Y\circ \mathcal{L}_X=\nabla_{[X,Y]},\quad\text{for all}\quad Y\in \mathfrak{X}(M),\]
where $\mathcal{L}_X$ denotes the Lie derivative. If the connection is flat affine, the previous equation gives that $X\in\mathfrak{a}(M,\nabla)$ if and only if
\begin{equation}\label{Eq:ecuacionkobayashi} \nabla_{\nabla_YZ}X=\nabla_Y\nabla_ZX,\quad\text{for all}\quad Y,Z\in \mathfrak{X}(M).  \end{equation}
The vector subspace $\it{aff}(M,\nabla)$ of $\mathfrak{a}(M,\nabla)$ whose elements are complete, with the usual bracket of vector fields,  is the Lie algebra of the group $\it{Aff}(M,\nabla)$  (see \cite{KoNo}). In this paper we use a simple fact about groups of homeomorphisms to determine complete infinitesimal affine transformations in some particular open sets of $\mathbb{R}^n$.\\
Let $p:\widehat{M}\rightarrow M$ denote the universal covering map of a real $n$-dimensional flat affine connected manifold $(M,\nabla)$. The pullback $\widehat{\nabla}$  of $\nabla$ by $p$ is a flat affine structure on $\widehat{M}$ and $p$ is an affine map. Moreover, the group $\pi_1(M)$ of deck transformations acts on $\widehat{M}$ by affine transformations.  There exists an affine immersion $D:(\widehat{M},\widehat{\nabla})\longrightarrow (\mathbb{R}^n,\nabla^0)$, called the developing map of $(M,\nabla)$, and a  group homomorphism  $A:\it{Aff}(\widehat{M},\widehat{\nabla}) \longrightarrow \it{Aff}(\mathbb{R}^n,\nabla^0)$ so that the following diagram commutes
\begin{equation}\label{Eq:diagramadeEhresman} \xymatrix{ \widehat{M} \ar[d]_{F} \ar[r]^{D} &\mathbb{R}^n\ar[d]^{A(F)}\\
\widehat{M} \ar[r]^{D} &\mathbb{R}^n.} \end{equation}
The map $D$ is called  the developing map and it was introduced by Ehresman (see \cite{E}). In particular for every $\gamma\in\pi_1(M)$, we have $D\circ \gamma= h(\gamma)\circ D$ with $h(\gamma):=A(\gamma)$. The map $h$ is also a group homomorphism called the holonomy representation of $(M,\nabla)$.   
	
If $M=G$ is an $n$-dimensional Lie group, a connection $\nabla$ is called left invariant if left multiplications, i.e., maps $L_g$   of $G$, with $g\in G$, defined by $L_g(h)=gh$, are affine transformations. Left invariant connections are usually denoted as $\nabla^+$. A Lie group $G$ endowed of a flat affine and left invariant connection $\nabla^+$ is called a flat affine Lie group. It is known that $(G,\nabla^+)$ is a flat affine Lie group if and only if there exists a homomorphism $\rho:\widehat{G}\rightarrow\it{Aff}(\mathbb{R}^n)$ whose corresponding action leaves an open orbit with discrete isotropy, see \cite{Kos} and \cite{Med}. Such representations are called affine \'etale representations. The open orbit turns out to be the image of the developing map $D:\widehat{G}\rightarrow \mathbb{R}^n$ and $D$ is a covering map of $\mathcal{O}$. Having a left invariant linear connection $\nabla^+$ on a Lie group $G$ is equivalent to have a bilinear product on $\mathfrak{g}=$Lie$(G)$ given by $X\cdot Y=(\nabla^+_{X^+}Y^+)_\epsilon$, where $\epsilon$ denotes the neutral element of $G$ and $X^+,Y^+$ are the left invariant vector fields on $G$ determined respectively by $X$ and $Y$. When the bilinear product is given, the connection is defined by $\nabla_{X^+}Y^+=(X\cdot Y)^+$ forcing the conditions $\nabla_{fX^+}Y^+=f\nabla_{X^+}Y^+$ and $\nabla_{X^+}fY^+=X^+(f)Y^+f\nabla_{X^+}Y^+$. That $\nabla^+$ is torsion free is equivalent to have 
\begin{equation}\label{Eq:torsionfree}
[X,Y]=X\cdot Y-Y\cdot X
\end{equation} and the connection is flat if and only if $[X,Y]\cdot Z=X\cdot(Y\cdot Z)-Y\cdot (X\cdot Z)$. Finally, the connection $\nabla^+$  is flat affine  if and only if the bilinear product is left symmetric, that is, 
\[ (X\cdot Y)\cdot Z-(Y\cdot X\cdot) Z=X\cdot(Y\cdot Z)-Y\cdot (X\cdot Z). \]
In this case, the algebra $(\mathfrak{g},\cdot)$ is called a left symmetric a algebra (see \cite{Med}). It is also known that the Lie algebra $\mathfrak{g}=$Lie$(G)$ is endowed of an associative product compatible with the Lie bracket if and only if the left invariant connection $\nabla^+$ is bi-invariant. That is, right multiplication on $G$, i.e., maps defined by $R_g(h)=hg$, for any $g$ and $h$ in $G$, are also affine transformations of $G$ relative to $\nabla^+$ (see \cite{Med}).

\section{Transformations preserving connected open sets}

In this section we show that, for connected open sets of $\mathbb{R}^n$, a classical affine transformation preserves the open set if and only if it preserves its boundary. We present some examples where the result is useful.

\begin{lemma} If $\0$ is an open set of $\mathbb{R}^n$ with boundary $\partial\0$ then $$\{T\in \it{Aff}(\mathbb{R}^n)\mid T(\0)=\0 \}\subseteq \{T\in \it{Aff}(\mathbb{R}^n)\mid T(\partial \0)=\partial\0 \}.$$
\end{lemma}
\begin{proof} First notice that $\mathcal{A}=\{T\in \it{Aff}(\mathbb{R}^n)\mid T(\0)=\0 \}$ and $ \{T\in \it{Aff}(\mathbb{R}^n)\mid T(\partial \0)=\partial\0 \}$ are groups. Now, take  $T\in \it{Aff}(\mathbb{R}^n)$ so that $T(\0)=\0$ and let  $ x\in\partial\0 $. Choose a convergent sequence $(x_n)\subseteq \0$ with $ x=\displaystyle{\lim_{n\rightarrow\infty}} x_n $. Hence $ T(x)=\displaystyle{\lim_{n\rightarrow\infty}} T(x_n) $ and therefore $ T(x)\in \overline{\0}$. However    $ T(x)\notin\0 $, otherwise since $T^{-1}\in\mathcal{A}$, we would have that  $x=T^{-1}\circ T(x)\in\0$. It follows that $T(\partial\0)\subseteq\partial\0$. As $T^{-1}\in\mathcal{A}$, the previous argument shows that $T^{-1}(\partial\0)\subseteq\partial\0$ and by applying $T$, the other inclusion follows.\end{proof}

\begin{lemma} \label{T:preservarorbitaigualapreservarfrontera} If  $ \0 $ is a connected open set of $\mathbb{R}^n$,  then we have
	$$ \llla T\in \it{Aff}(\R^n) \mid T(\0) =\0\rlla=\llla  T\in \it{Aff}(\R^n) \mid T(\partial\0)=\partial \0 , T(p)\in\0 \text{ for some}\ p\in \0 \rlla .$$
\end{lemma}
\begin{proof} The first inclusion follows from the previous lemma. For the other inclusion, let $ T\in\mathcal{B}=\llla  T\in \it{Aff}(\R^n) \mid T(\partial\0)=\partial \0 , T(p)\in\0 \text{ for some}\ p\in \0 \rlla $ and let $ q\in\0 $. Since $\0$ is connected, it is path connected, so there exists a  continuous path $s$ from $ q $ to $p$ totally contained in $\0$. Hence $ T\circ s $ is a continuous path  from $ T(q) $ to $ T(p) $ and we claim that $ T(q)\in\0 $. Otherwise, as $ T(p)\in\0 $ there should exists a point $ q' $ on the path $ s $ so that $ T(q')\in\partial\0 $. But since $T^{-1}$ preserves $\partial\0$, we would have that $ q'\in\partial\0 $  contradicting the choice of $s$. This proves that $T(\0)\subseteq \0$.  
	
A similar argument considering a continuous path  from $q$ to $T(p)$, shows that $T^{-1}(q)\in\0$ whenever $q\in\0$. That is, $T^{-1}(\0)\subseteq\0$. By applying $T$, we get that $\0\subseteq T(\0)$ \end{proof}
	
	%	Now take $q\in\0$, we claim that $T^{-1}(q)\in\0$. For this, consider the continuous path $s$ formed by a sequence of line segments from $q$ to $T(p)$. Hence $T^{-1}(s)$ is a continuous path formed by line segments from $T^{-1}(q)$ to $p$. If $T^{-1}(q)\notin\0$, there would exists a point $y\in T^{-1}(s)\cap\partial\0$, but since $T^{-1}$ preserves $\partial\0$, we get that $T(y)\in \partial\0\cap s$ contradicting the choice of $s$.  
%\end{proof}
%\vspace{1cm}
\begin{remark} The previous lemmas hold true for homeomorphisms of $\mathbb{R}^n$.
\end{remark}

Now we present two examples. First recall that an affine frame on the affine space $\mathbb{R}^n$ is a set of $n+1$ points $p_0,p_1,\cdots,p_n$ so that the vectors $\overrightarrow{p_0p_1},\dots,\overrightarrow{p_0p_n}$ form a linear basis of $\mathbb{R}^n$ seen as a vector space.

\begin{example} Let $p_0,p_1,\cdots,p_n$ be an affine frame on the $n$-dimensional affine space $\mathbb{R}^n$ and consider the manifold   $M_i=\mathbb{R}^n\setminus\{p_0,\dots,p_{i-1}\}$ endowed with the connection $\nabla^i$  given by the restriction to $M_i$ of the usual connection  $\nabla^0$. Also set $v_k=\overrightarrow{p_0p_k}$, for $k=0,\dots n$ and consider the basis $\beta=\{v_1,\dots,v_n\}$ of $\mathbb{R}^n$.
	
By the previous lemma, the group $\it{Aff}(M_1,\nabla^1)$ is given by affine transformations of $\mathbb{R}^n$ fixing the point $p_0$. Hence, seen as transformations of $\mathbb{R}^n$, they are  determined by the linear part. That is, $\it{Aff}(M_1,\nabla^1)$  is locally isomorphic to $GL(\mathbb{R}^n)$.
	
It is easy to self convince that, in the space of affine transformations fixing a set of two points $\{A,B\}$, there is no continuous path from one transformation fixing both points to a transformation permuting them. From this observation and the previous lemma, the connected component $\it{Aff}(M_{i+1},\nabla^{i+1})_0$   containing the unit of the group $\it{Aff}(M_{i+1},\nabla^{i+1})$ is locally isomorphic to the group of affine transformations of $\mathbb{R}^n$ fixing the points $p_0,\dots,p_i$ and whose linear part is orientation preserving.  Since every element of  $\it{Aff}(M_{i+1},\nabla^{i+1})_0$  fixes $p_0$, every $T\in \it{Aff}(M_{i+1},\nabla^{i+1})_0$ is determined by its  linear part $L$ and this linear transformation fixes the vectors $v_1,\dots,v_i$. Thus, the matrix of $L$ with respect to the basis $\beta$ is of the form
	\[ [L]_\beta=\left[ \begin{matrix}  I&B\\0&A  \end{matrix} \right] \] 
with $I$ the identity matrix of size $i\times i$. Hence, for $i\leq n$, the group $\it{Aff}(M_{i+1},\nabla^{i+1})_0$ is locally isomorphic to the matrix Lie group $\left\{ \left[ \begin{matrix}  I&B\\0&A  \end{matrix} \right]\ \bigg|\ \det A>0  \right\}.$ This group is also locally isomorphic to the matrix Lie group $\left\{ \left[ \begin{matrix}  A&B\\0&I  \end{matrix} \right]\ \bigg|\ \det A>0  \right\}$.

\end{example}

\begin{remark} \label{R:planewithiholes} The group $\it{Aff}(M_{i+1},\nabla^{i+1})_0$ of the previous example is locally isomorphic to the group $(\mathbb{R}^{n-i})^i\rtimes_\theta GL(\mathbb{R}^{n-i})$ where $\theta(A)(w_1,\dots,w_i)=(Aw_1,\dots,Aw_i)$. In particular, the group $\it{Aff}(M_{2},\nabla^{2})$ is locally isomorphic to $\it{Aff}(\mathbb{R}^{n-1})$. Also notice that the group $\it{Aff}(M_{n+1},\nabla^{n+1})_0$ is trivial as the group $\it{Aff}(M_{n+1},\nabla^{n+1})$ is discrete with elements determined by affine transformations of $\mathbb{R}^n$ permuting the set of points $p_0,\dots,p_n$. This follows directly from the fact that an affine transformation is uniquely determined by its values on a frame (see \cite{Ber}), just as a linear transformation is uniquely determined by its values in a  basis. 
	
As a consequence, for $i=0,\dots n$, the Lie algebra of the group $\it{Aff}(M_{i+1},\nabla^{i+1})$ is isomorphic to the  algebra	of matrices of the form $\left[ \begin{matrix}  A&B\\0&0  \end{matrix} \right]$, with $A$ and $B$ of sizes $(n-i)\times (n-i)$ and $i\times (n-i)$, respectively. Finally, notice that this algebra is endowed with an associative product compatible with its Lie bracket, namely, the usual product of matrices.
\end{remark}

\begin{example} \label{Ex:hyperplanes} Consider the open set of $ \R^n $ given by $\0=\llla (x_1,x_2, \dots , x_n) \mid x_1>0 \rlla $. According to Lemma \ref{T:preservarorbitaigualapreservarfrontera}, to determine the elements in $ \it{Aff}(\R^n) $ preserving $\0$, it is enough to find all affine transformations fixing its boundary, i.e.,  $ T(\partial\0)=\partial\0 $, hence we need  conditions so that the following equation holds true
\[ 	\left[ \begin{matrix}
	a_{11} & \cdots & a_{1n} & a_{1,n+1} \\ 
	\vdots & \ddots & \vdots & \vdots \\
	a_{n1} & \cdots & a_{nn} & a_{n+1,n+1}
	\end{matrix} \right]
	\left[ \begin{matrix}
	0\\x_2\\\vdots\\x_n
	\end{matrix}\right]
	=\left[ \begin{matrix}
	0\\x'_2\\\vdots\\x'_n
	\end{matrix}\right]. \]
	A simple computation gives that $ a_{1j}=0 $ para $ 1<j\leq n+1 $. To satisfy the condition that $T(p)\in\0$ for some $p\in\0$ we also need that $ a_{11}>0 $. Therefore, the group of transformations preserving  $ \0 $ is locally isomorphic to the group of affine transformation of the form
	\[
	\llla A=\left[ \begin{matrix}
	a_{11} &   0    & \cdots & 0&     0      \\
	a_{21} & a_{22} & \cdots& a_{2n} &  a_{2,n+1}  \\
	\vdots & \vdots& \ddots & \vdots &   \vdots    \\
	a_{n1} & a_{n2} & \cdots & a_{nn} & a_{n,n+1}
	\end{matrix} \right] \bigg|\ A\text{ invertible and } a_{11}>0\rlla
	\]
	where the first $n$ columns give the linear part and the last column is the translation part. \\
	More generally, if $\0_i=\llla (x_1,x_2, \dots , x_n) \mid x_1>0,\ x_2>0, \dots,\ x_i>0 \rlla $  and $\nabla^i$ the connection $\nabla^0$ in $\0_i$, with $ 1\leq i \leq n $, the connected component $\it{Aff}(\0_i,\nabla^i)_0$ of the group of transformations $\it{Aff}(\0_i,\nabla^i)$ preserving this orbit $ \0_i $ is given by affine transformations of the form
	\begin{equation} \label{Eq:groupofaffinetransformationsoftheintersectionofhalfplanes}
	\llla B=\left[ \begin{smallmatrix}
	a_{11} &    \cdots &  0 & 0 & \cdots & 0 & 0   \\
	\vdots &  \ddots & \vdots & \vdots & \ddots &\vdots & \vdots\\
	0    & \cdots &    a_{ii} & 0 & \cdots & 0 & 0\\
	a_{i+1,1} & \cdots & a_{i+1,i}  & a_{i+1,i+1} & \cdots & a_{i+1,n}  & a_{i+1,n+1} \\
	\vdots & \ddots & \vdots & \vdots & \ddots & \vdots& \vdots\\
	a_{n1}  &\cdots & a_{ni} & a_{n,i+1} & \cdots & a_{n,n} & a_{n,n+1}
	\end{smallmatrix} \right]\bigg|\ B\text{ invertible and } a_{11}, \dots, a_{ii}>0 \rlla
	\end{equation}
where  the first $n$ columns denote the linear part and the last column the translation part.  
\end{example}

For  orbits as in the previous example we have.

\begin{lemma} \label{l:orbitbyhyperplanes} If $\0_i$ and $\nabla^i$ are as in the previous example, the group $\it{Aff}(\0_i,\nabla^i)$   is locally isomorphic to
	\begin{equation}\label{Eq:group reduced} \it{Aff}(\0_i,\nabla^i)_0=\left(( (\mathbb{R}^{n-i})^i\rtimes_{\theta_1}\it{Aff}(\mathbb{R}^{n-i}) )^{op}\rtimes _{\theta_2}(\mathbb{R}^{>0})^i\right)^{op}  \end{equation}
where the superindex $op$ denotes the opposite Lie group	and   $\theta_1$ and $\theta_2$ are, respectively, the action of $\it{Aff}(\mathbb{R}^{n-i})$ on $(\mathbb{R}^{n-i})^i$ and the action of $(\mathbb{R}^{>0})^i$ on $(\mathbb{R}^{n-i})^i\rtimes_{\theta_1}\it{Aff}(\mathbb{R}^{n-i})$ defined by 
	\begin{align*}\theta_1&(A,w)(v_1,\dots,v_i)=(Av_1,\dots,Av_i)\quad\text{and}\\ 
	\theta_2&(\lambda_1,\dots,\lambda_i)(v_1,\dots,v_i,A,w)=(\lambda_1 v_1,\dots,\lambda_iv_i,A,w) \end{align*}
Morevoer, its Lie algebra is associative
\end{lemma}
\begin{proof} An easy calculation shows that the map between the groups of Equations  \eqref{Eq:groupofaffinetransformationsoftheintersectionofhalfplanes} and \eqref{Eq:group reduced} defined by $\left[\begin{matrix} D&0&0\\A&B&w\end{matrix} \right]\mapsto ((A,B,w),D)$ is an isomorphism, where the matrix on the left denotes an element on the group in Equation \eqref{Eq:groupofaffinetransformationsoftheintersectionofhalfplanes} with $B$ invertible and $D$ a diagonal matrix of positive determinant. 
	
The Lie algebra  $\it{aff}(\0_i,\nabla^i)$ is the  algebra of all matrices  of the form $\left[\begin{matrix} D&0&0\\A&B&w\\0&0&0\end{matrix} \right]$, with $D$ diagonal. Hence it is an associative algebra with the usual product of matrices. 
\end{proof}

\section{Applications and examples}

Given a manifold $M$ endowed of a linear connection, for any $p\in M$, there exists a neighborhood $N_p$ of $0\in T_pM$ so that for any $X_p\in N_p$, the geodesic $\gamma_t$ with initial condition $(p,X_p)$ is defined in an open interval $(-\epsilon,\epsilon)$ with $\epsilon>1$. The map
\[ \begin{matrix} exp_p^\nabla:&T_pM&\rightarrow& M\\ &X_p&\mapsto&\gamma_1\end{matrix} \]
is called the exponential map relative to $\nabla$. 

As an application of the results in the previous section, we get the following.

\begin{theorem} \label{T:mainresult} Let $ (M,\nabla) $ be a flat affine simply connected manifold. If the domain of the exponential map relative to $\nabla$ at some $p\in M$ is a convex subset of $T_pM$, then the group of affine transformations $ \it{Aff}(M,\nabla) $ is locally isomorphic to $ \llla T\in \it{Aff}(\R^n) \mid T(\0) =\0\rlla $, with $\0$ the image under the developing map. It is also locally isomorphic to  $$\llla  T\in \it{Aff}(\R^n) \mid T(\partial\0)=\partial \0 ,\ T(p)\in\0\ \text{for some}\ p\in \0 \rlla. $$
More generally, the conclusion is true when the developing map $D:M\rightarrow \0$ is a diffeomorphism onto $\0$.
\end{theorem}
\begin{proof} As $M$ is simply connected and the domain of the exponential is convex, the developing map $D:(M,\nabla)\rightarrow (\0,\nabla^0)$ is an affine diffeomorphism (see \cite{Kos}, see also \cite{Shi}, page 151), hence  the result follows from  the commutativity of the diagram \eqref{Eq:diagramadeEhresman}.\end{proof}
	
\begin{example} \label{Ex:connectionsontheplane} In $\mathbb{R}^2$ there are, up to isomorphism, six flat affine connections whose group of affine transformations act transitively on $\mathbb{R}^2$ (see \cite{Ben}).  The connections are left invariant determined by the following affine \'etale representations
	\[ \begin{array}{lll}
	\rho_1(a,b)=\begin{bmatrix}
	1 & 0 & a \\
	0 & 1 & b \\
	\end{bmatrix},\quad & \rho_2(a,b)=\begin{bmatrix}
	1 & b & a+\frac{1}{2}b^2 \\
	0 & 1 & b \\
	\end{bmatrix} ,\quad &
	\rho_3(a,b)=\begin{bmatrix}
	1 & 0 & a \\
	0 & e^b & 0 \\
	\end{bmatrix},\\\\ \rho_4(a,b)=\begin{bmatrix}
	e^a & be^a & 0 \\
	0 & e^a & 0 \\
	\end{bmatrix},\quad &
	\rho_5(a,b)=\begin{bmatrix}
	e^a & 0 & 0 \\
	0 & e^b & 0 \\
	\end{bmatrix},\quad & \rho_6(a,b)=e^a\begin{bmatrix}
	\ \ \cos b & \sin b & 0 \\
	-\sin b & \cos b & 0 \\
	\end{bmatrix}
	\end{array} \]
where the first two columns of the images give the linear part and the third column, the translation part. The corresponding developing maps are defined by 
\[ \begin{array}{lclclc} D_1(x,y)=(x,y),&D_2(x,y)=\left(x+\frac{y^2}{2},y\right),& D_3(x,y)=(x,e^y),\\ D_4(x,y)=(e^x,ye^x),& D_5(x,y)=(e^x,e^y),&  D_6(x,y)=(e^x\cos y,e^x\sin y).\end{array} \]
Hence the orbit determined by the action corresponding to the representations $\rho_1$  and $\rho_2$ is the whole plane, the orbit determined by $\rho_3$ and $\rho_4$ is the upper half plane $\{(x,y)\mid y>0\}$, the representation $\rho_5$ leaves the first cuadrant $\{ (x,y)\mid x,y>0 \}$ as the open orbit and for the action $\rho_6$, the orbit is the punctured plane $\{ (x,y)\mid (x,y)\ne(0,0) \}$. As the developing maps $D_1$ to $D_5$ are diffeomorphisms onto the respective orbit, by denoting by $\nabla^i$, $i=1,\dots,6$ the  left invariant connection on $\mathbb{R}^2$ determined by $\rho_i$, from  Lemma \ref{T:preservarorbitaigualapreservarfrontera} and  Theorem \ref{T:mainresult},  we get
\begin{align*}
\it{Aff}(\mathbb{R}^2,\nabla^1)_0&=\it{Aff}(\mathbb{R}^2,\nabla^2)_0\cong \it{Aff}(\mathbb{R}^2,\nabla^0)\\
\it{Aff}(\mathbb{R}^2,\nabla^3)_0&=\it{Aff}(\mathbb{R}^2,\nabla^4)_0\cong\left\{\left[ \begin{matrix} a&b&c\\0&d&0 \end{matrix}\right]\bigg|\ a\ne0,\ d>0\right\}\\
\it{Aff}(\mathbb{R}^2,\nabla^5)_0&\cong\left\{\left[ \begin{matrix} a&0\\0&b \end{matrix}\right]\bigg|\ a,b>0\right\}\cong (\mathbb{R}^{>0})^2\end{align*}
Although Theorem \ref{T:mainresult} does not apply in the last case, it is  known that $\it{Aff}(\mathbb{R}^2,\nabla^6)_0\cong GL_2(\mathbb{R})^+$ (see \cite{Nag}), where the $+$ is used to denote linear transformations preserving orientation.
\end{example}

\begin{example} Consider the group $G=\it{Aff}(\mathbb{R})_0$ given by classical affine orientation preserving transformations of $\mathbb{R}$, known as the group of  affine motions of the line (orientation preserving). This group is isomorphic to $\mathbb{R}^{>0}\times \mathbb{R}$ with the product $(a,b)(c,d)=(ac,ad+b)$. Its Lie algebra is isomorphic to $\mathfrak{g}=$Lie$(G)\cong\mathbb{R}e_1\oplus\mathbb{R}e_2$ with the bracket $[e_1,e_2]=e_2$. Up to isomorphism, there are two families and four exceptional  left symmetric products compatible with the bracket on $\mathfrak{g}$ (see \cite{MSG})
\begin{center}
 \begin{tabular}{c|cc} $\mathcal{F}_1(\alpha)$ &$e_1$&$e_2$ \\\hline $e_1$&$\alpha e_1$ & $e_2$ \\ $e_2$&$0$&$0$ 
\end{tabular}$\qquad$ \begin{tabular}{c|cc} $\mathcal{F}_2(\alpha)$ $(\alpha\ne0)$ &$e_1$&$e_2$ \\\hline $e_1$&$\alpha e_1$ & $(\alpha+1)e_2$ \\ $e_2$& $\alpha e_2$&$0$ 
\end{tabular}\\   \begin{tabular}{c|cc} $\mathcal{E}_1$ &$e_1$&$e_2$ \\\hline $e_1$&$ e_1+e_2$ & $e_2$ \\ $e_2$&$0$&$0$ 
\end{tabular}$\qquad$ \begin{tabular}{c|cc}$\mathcal{E}_2$ &$e_1$&$e_2$ \\\hline $e_1$&$ -e_1+e_2$ & $0$ \\ $e_2$&$-e_2$&$0$ \end{tabular}$\qquad$ \begin{tabular}{c|cc} $\mathcal{E}_{3}$ and $\mathcal{E}_{4}$ &$e_1$&$e_2$ \\\hline $e_1$&$ 2e_1$ & $e_2$ \\ $e_2$&$0$&$\pm e_1$ \end{tabular}
\end{center}
where $\alpha$ is a real parameter. Denoting by $F_{i,(\alpha)}$ and $E_j$ the corresponding developing maps determined by the products $\mathcal{F}_i(\alpha)$ and $\mathcal{E}_j$, respectively, for $i=1,2$ and $j=1,\dots,4$ one can check that
\[ \begin{array}{lcclc} F_{1,(\alpha)}(x,y)=\left(\frac{1}{\alpha}x^\alpha,y\right)\text{ for }  \alpha\ne0
&\ & F_{1,(0)}(x,y)=(\ln x,y)\\ &\ & \\
F_{2,(\alpha)}(x,y)=\left(\frac{1}{\alpha}x^\alpha,x^\alpha y\right)\text{ for }  \alpha\ne\{0,-1\}&\ & F_{2,(-1)}(x,y)=\left(-\frac{1}{x},\frac{y}{x}\right)\\ &\ & \\ 
E_1(x,y)=(x,1+x+y+x\ln x)&\ & E_2(x,y)=\left(-\frac{1}{x},\frac{1}{x}+\frac{y}{x}+\ln x-1\right)\\  &\ & \\
E_3(x,y)=\left(\frac{1}{2}(x^2+y^2-1),y\right)&\ & E_4(x,y)=\left(\frac{1}{2}(x^2-y^2-1),y\right)
\end{array} 
  \]
It is easy to check that the image of $F_{1,(0)}$ is the whole plane, the image under the maps  $F_{1,(\alpha)}$, $F_{2,(\alpha)},$ $E_1,$ and $E_2$ is the right half plane $\0=\{ (x,y)\mid x>0\}$, and the image under the maps $E_3$ and $E_4$ is the interior of the parabola $x=\frac{1}{2}(y^2-1)$. It is also easy to see that these developing maps are diffeomorphisms, hence by applying Lemma \ref{T:preservarorbitaigualapreservarfrontera},  Theorem \ref{T:mainresult},  and using Example \ref{Ex:hyperplanes} we get 
\begin{align*}
\it{Aff}(G,\nabla_1^+(0))&\cong \it{Aff}(\mathbb{R}^2)\\
\it{Aff}(G,\nabla_i^+(\alpha))&\cong\it{Aff}(G,\nabla_i^+)\cong \left\{\left[\begin{matrix}a&0&0\\b&c&d
\end{matrix}\right]\bigg|\ a>0\text{ and }c\ne0\right\}, \ i=1,2\text{ and }\alpha\ne0\\
\it{Aff}(G,\nabla_3^+)&\cong  \left\{\left[\begin{matrix}a^2& ab&(a^2+ b^2-1)/2\\0&a&b
\end{matrix}\right]\bigg|\ a>0\right\} \\
\it{Aff}(G,\nabla_4^+)&\cong  \left\{\left[\begin{matrix}a^2&-ab&(a^2- b^2-1)/2\\0&a&b
\end{matrix}\right]\bigg|\ a>0\right\}
\end{align*}
where  $\nabla_i^+(\alpha)$ and $\nabla_j^+$ denote the corresponding left invariant flat affine connections on $G$ determined by the products $\mathcal{F}_i(\alpha)$ and $\mathcal{E}_j$, respectively, for $i=1,2$ and $j=1,\dots,4$,
\end{example}

We finish this work with the following.

\begin{corollary} If $(M,\nabla)$ is a flat affine simply connected manifold so that the  developing map is an isomorphism onto  	$\0_i=\{(x_1,\dots,x_n)\mid x_1>0,\dots,x_i>0\}$ or $\0_i$ equal to the space $\mathbb{R}^n$ with $i$ holes, $0\leq i\leq n$, then the group $\it{Aff}(M,\nabla)$ admits a flat affine bi-invariant connection.
\end{corollary}
\begin{proof} From Remark \ref{R:planewithiholes} and Lemma \ref{l:orbitbyhyperplanes} the Lie algebra $\it{aff}(\0_i,\nabla^i)$ of the group $\it{Aff}(\0_i,\nabla^i)$   admits an associative product compatible with the Lie bracket. Hence, it determines a flat affine bi-invariant connection on $\0_i$ (see \cite{Med}). Since  Theorem \ref{T:mainresult} implies that the group $\it{Aff}(M,\nabla)$ is locally isomorphic to $\it{Aff}(\0_i,\nabla^i)$, we get that $\it{Aff}(M,\nabla)$ admits a flat affine bi-invariant connection. 
\end{proof}

\begin{example}
 If $\nabla$ is any of the connections $\nabla^1$ to $\nabla^5$ on $\mathbb{R}^2$ given in Example \ref{Ex:connectionsontheplane}, the group of affine transformations $\it{Aff}(\mathbb{R}^2,\nabla^i)$ admits a flat affine bi-invariant connection. The group $\it{Aff}(\mathbb{R}^2,\nabla^6)$ also admits a flat bi-invariant connection as it is locally isomorphic to $GL_2(\mathbb{R}^2)^+$ and its Lie algebra, $gl_2(\mathbb{R}^2)$, admits an associative product compatible with the Lie bracket. The corollary also applies to the group $\it{Aff}(\it{Aff}(\mathbb{R})_0,\nabla^+)$ of affine transformations of $G=\it{Aff}(\mathbb{R})_0$ preserving any left invariant connection $\nabla^+$ on $G$.
\end{example}
% \it{Aff}(\mathbb{R}^2,\nabla^0)$.

\noindent{\bf Conflict of interest.} On behalf of all authors, the corresponding author states that there is no conflict of interest.

\end{document}